\documentclass[draftcls, onecolumn, 11pt]{IEEEtran}

\usepackage{amsfonts,comment,amsthm}
\usepackage{amssymb,amsmath,url,bm}
\usepackage{mathrsfs}
\usepackage{hyperref}
\usepackage{graphicx,color,float,epstopdf}

\usepackage{pst-blur,pstricks-add}
\usepackage{cases}
\usepackage{algorithm}
\usepackage{algorithmicx}
\usepackage{algpseudocode}
\usepackage{amsfonts, bm}
\usepackage{graphics,booktabs,color,epsfig,subfigure}
\usepackage[numbers,sort&compress]{natbib}
\usepackage{amsmath,amssymb,amsfonts}
\usepackage{graphicx}
\usepackage{textcomp}
\usepackage{xcolor}
\usepackage{subfigure}
\usepackage{multirow}

\newtheorem{lemma}{Lemma}

\newtheorem{theorem}{Theorem}
\newtheorem{Def}{Definition}

\newtheorem{example}{Example}
\newtheorem{remark}{Remark}
\newcommand{\h}[1]{\mathbf{#1}}
\DeclareMathOperator{\sign}{sign}

\newcommand{\nn}{\nonumber}

\begin{document}
\bibliographystyle{IEEEtran}
% paper title
\title{ On NP-Hardness of $L_1/L_2$ Minimization and Bound Theory of Nonzero Entries in Solutions }

\author{Min Tao, Xiao-Ping  Zhang,~\IEEEmembership{Fellow,~IEEE} and Yun-Bin Zhao
\thanks{The work of M. Tao is supported partially by the Natural Science Foundation of China (No.
12471289,\ 12371318) and partially by
Jiangsu University QingLan Project. The work of X.-P. Zhang was supported in part by Shenzhen Key Laboratory of Ubiquitous Data Enabling (No. ZDSYS20220527171406015), and by Tsinghua Shenzhen International Graduate School-Shenzhen Pengrui Endowed Professorship Scheme of Shenzhen Pengrui Foundation. The work of Y.-B. Zhao was partially supported by the National Natural Science Foundation of China
(No. 12071307, 12471295) and Hetao Shenzhen-Hong Kong Science and Technology  Innovation Cooperation Zone Project (HZQSWS-KCCYB-2024016)}
\thanks{M. Tao is with the School of Mathematics, National Key Laboratory for Novel Software Technology, Nanjing University, Nanjing, 210093, Republic of China. Email: taom@nju.edu.cn}
\thanks {X.-P. Zhang is with the Shenzhen Key Laboratory of Ubiquitous Data Enabling, Shenzhen International Graduate School, Tsinghua University, Shenzhen 518055, China. Email: xpzhang@ieee.org}
\thanks{Y.-B. Zhao is with the Shenzhen International Center for Industrial and Applied Mathematics, SRIBD, Chinese University of Hong Kong, Shenzhen 518116, China. Email: yunbinzhao@cuhk.edu.cn}
}

\maketitle

\begin{abstract}
%The \(L_1/L_2\) norm ratio minimization has gained significant attention as a sparsity  prompting approach, and
%found many successful applications across various domains.
%A variety efficient algorithms have been developed to computing the stationary points of these $L_1/L_2$ minimization problem.
%However, their computational complexity remains open. In this paper, we unexpectedly reveal that
% finding the global minimum of both constrained and unconstrained \(L_1/L_2\) models is strongly NP-Hard.
%Simultaneously, we establish uniform upper bounds in $L_2$ norm for any local minimizer of constrained and unconstrained \(L_1/L_2\)-minimization models. Furthermore, we derive some upper/lower bound for the magnitudes of nonzero entries in any local minimizer of the unconstrained model
%to help classify nonzero entries to prompt sparsity.
%Moreover, we point out that finding the global minimum of constrained and unconstrained \(L_p\) (\(0 < p \leq 1\)) over \(L_q\) (\(1 < q < +\infty\)) models is also strongly NP-Hard.

The \(L_1/L_2\) norm ratio has gained significant attention as a measure of sparsity due to three merits: sharper approximation to the \(L_0\) norm compared to the  \(L_1\) norm,  being parameter-free and scale-invariant, and exceptional performance with highly coherent matrices. These properties have led to its successful application across a wide range of fields. While several efficient algorithms have been proposed to compute stationary points for \(L_1/L_2\) minimization problems, their computational complexity has remained open. In this paper, we prove that finding the global minimum of both constrained and unconstrained \(L_1/L_2\) models is strongly NP-hard.

In addition, we establish uniform upper bounds on the \(L_2\) norm for any local minimizer of both constrained and unconstrained \(L_1/L_2\) minimization models. We also derive upper and lower bounds on the magnitudes of the nonzero entries in any local minimizer of the unconstrained model, aiding in classifying nonzero entries. Finally, we extend our analysis to demonstrate that the constrained and unconstrained \(L_p/L_q\) (\(0 < p \leq 1, 1 < q < +\infty\)) models are also strongly NP-hard.

 %In particular, ADMM$_p^+$ reduces computational time by about $95\%\sim99\%$ while achieving a much higher accuracy than the commonly used  scaled gradient projection method for the wavelength misalignment problem.
\end{abstract}
\begin{IEEEkeywords}
	Nonconvex programming, global optimization,
sparse signal recovery, NP-hardness, computational complexity
\end{IEEEkeywords}

%\begin{keywords} Nonconvex programming, Global optimization,
%Sparse solution recovery
%\end{keywords}
% REQUIRED
	%\begin{AMS}
%	90C26,   90C51
%	\end{AMS}

	% REQUIRED
	\section{Introduction}
\label{RelatedW}
Compressive sensing (CS) aims to recover sparse signals using a set of undersampled linear measurements. Consequently, the fundamental problem in CS can be formulated as the following \( L_0 \)-minimization problem:

\begin{eqnarray}\label{L0}
\min _{{\h x} \in {\cal X}} \|{\h x}\|_{0}, \quad \text{subject to} \quad A {\h x} = {\h b},
\end{eqnarray}

\noindent where \( A \in \mathbb{R}^{m \times n} \) with \( m \ll n \) is a given sensing matrix, and \( {\h b} \in \mathbb{R}^m \) is the observation (measurement) vector (\( {\h b} \neq \mathbf{0} \)). Here, \( \|\cdot\|_{0} \) denotes the \( L_0 \) norm, which counts the number of nonzero entries in a vector, and \( {\cal X} \) represents either \( \mathbb{R}^n_+ \) or \( \mathbb{R}^n \). Problem (1) is known to be NP-hard \cite{BKNa95}, indicating that it cannot be solved in polynomial time as the problem size increases.

As a practical alternative, the \( L_1 \) penalty regularization can be employed, replacing the \( L_0 \) norm with the \( L_1 \) norm, resulting in the following optimization problem:

\begin{equation}\label{L1}
\min _{{\h x} \in {\cal X}} \|{\h x}\|_{1}, \quad \text{subject to} \quad A {\h x} = {\h b}.
\end{equation}

\noindent Several exact recovery conditions have been developed to ensure that the solution of problem (\ref{L1}) coincides with the solution of (\ref{L0}). These conditions include the exact \( L_1 \)-norm recovery condition, which states that a solution \( \h x_0 \) to \( A \h x_0 = \h b \) satisfies \( \|\h x_0\|_0 < \frac{1}{2}\left(1 + \frac{1}{\mu(A)}\right) \), where \( \mu(A) \) is the mutual coherence of \( A \) \cite{pnas.0437847100}. Other conditions include the restricted isometry property (RIP) \cite{1542412,6589953}, the null space condition \cite{AWR09}, and the restricted eigenvalue condition (ERC) \cite{1302316}, etc. The resulting model in (2) is convex and can be solved to global optimality in polynomial time \cite{Nesterov1994PA}.

While \( L_1 \) minimization is computationally tractable, it tends to bias toward larger coefficients \cite{FL01}, hindering true sparsity. To address this limitation, several nonconvex relaxation techniques have emerged, such as the \( L_p \) model (for \( 0 < p < 1 \)) \cite{CR07}, the transformed \( L_1 \) norm \cite{Nikolova00}, and the capped \( L_1 \) norm \cite{PELEG2008375}, all aimed at promoting better sparsity recovery. In particular, the \( L_1/L_2 \) ratio has gained attention due to its parameter-free and scale-invariant characteristic, providing a sharper approximation to \( L_0 \) norm  compared to the $L_1$ norm (see the landscapes of \( \|\h x\|_0 \), \( \|\h x\|_1 \), and \( \|\h x\|_1/\|\h x\|_2 \) in Fig. \ref{Landscape}) and demonstrating greater efficiency when dealing with highly coherent sensing matrices \cite{Tao20,TaoZhang23,YEX14,RWDL19,WYYL20}. As a result, \( L_1/L_2 \) minimization \cite{WCG78,YEX14} has found successful applications across various domains, highlighting its effectiveness and practicality \cite{6638838,KTF11,Audrey15,PHAM201796,JIA2018198}.

\begin{figure*}[htbp!]
	\vspace{-0cm}\centering{\vspace{0cm}
		\includegraphics[scale=0.28]{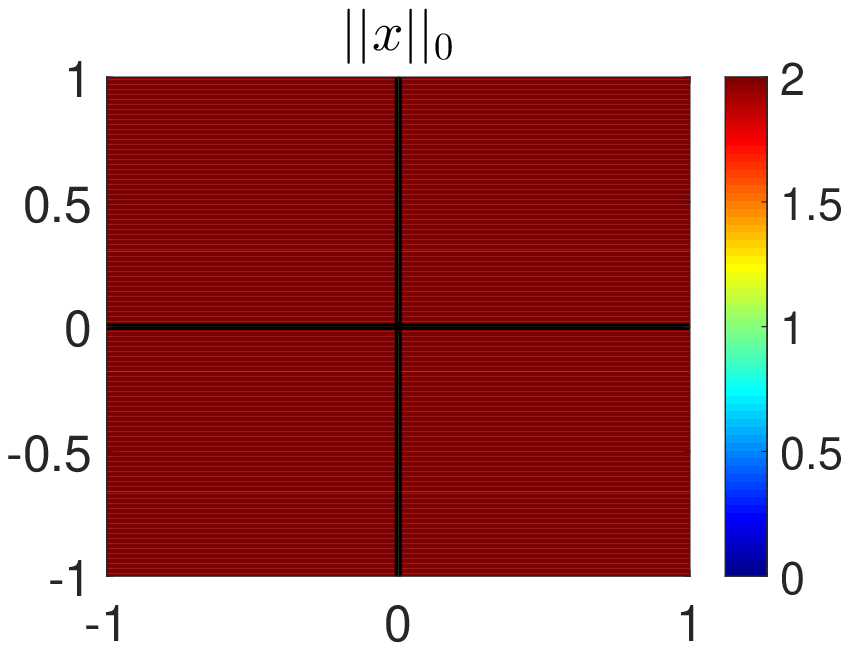}\!\!\!\!\!\!\!\!\!\!\!\!\!\!
			\includegraphics[scale=0.28]{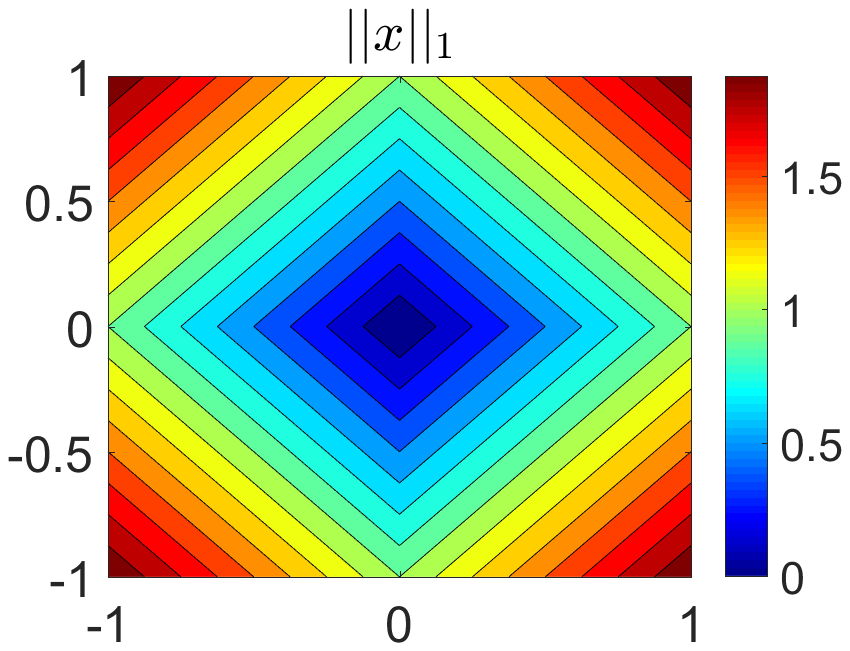}\!\!\!\!\!\!\!\!\!\!\!\!
			\includegraphics[scale=0.28]{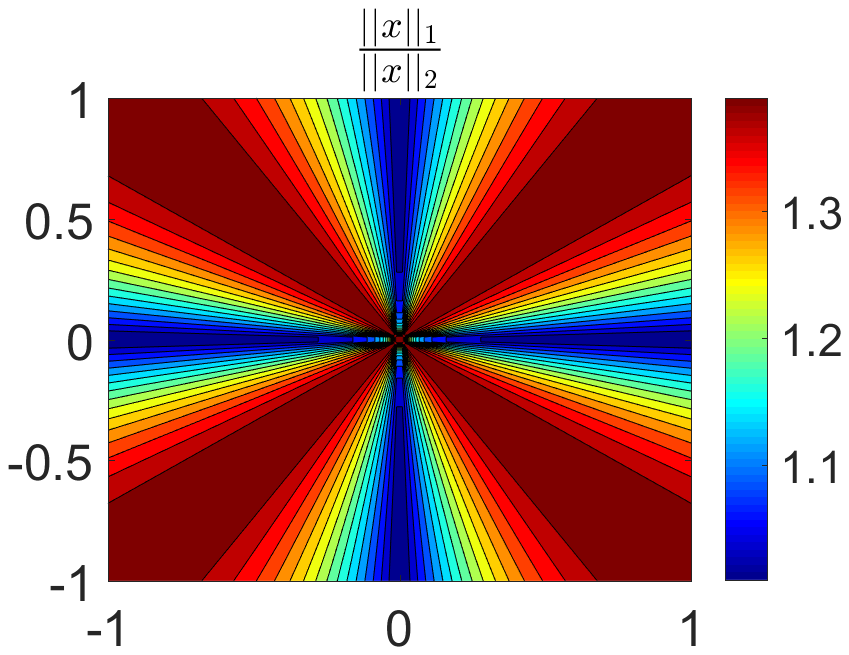}
	} \vspace{-0cm}\centering{\caption{The landscapes of different norm: From left to right: $\|{\h x}\|_0$, $\|{\h x}\|_1$ and $\frac{\|{\h x}\|_1}{\|{\h x}\|_2}$ in $\mathbb{R}^2$ space.
}}
	\label{Landscape}\end{figure*}
\subsection{Motivation and related works}
Various exact recovery conditions from \(L_1/L_2\) to \(L_0\) minimization have been studied by Yin et al.~\cite{YEX14}, Xu et al.~\cite{XU2021486}, and Zhou and Yu~\cite{ZHOU2019247}. The nonconvex and nonsmooth nature of \(L_1/L_2\) minimization introduces significant challenges in obtaining global solutions. Consequently, several algorithms have been developed to compute stationary points instead. For example, the alternating direction method of multipliers (ADMM) with double variable duplication \cite{RWDL19}, and the accelerated schemes \cite{WYYL20}, have been applied to the constrained  $L_1/L_2$ model.

On the other hand, methods such as the scaled gradient projection method (SGPM) \cite{ELX13, YEX14}, ADMM\(_p\) \cite{Tao20}, and ADMM\(_p^+\) \cite{TaoZhang23} are designed for the unconstrained $L_1/L_2$ model. The moving balls approximation algorithm \cite{ZengYuPong20} has been developed to address the constrained model in both noiseless and noisy scenarios.

Additionally, a DCA-type algorithm has been proposed to solve the sorted \(L_1/L_2\) minimization problems, which assign higher weights to indices with smaller absolute values and lower weights to larger values, thereby promoting sparsity \cite{Sorted24}. The proximal subgradient algorithm with extrapolation techniques \cite{BRDL22} is also employed to tackle nonconvex and nonsmooth fractional programming problems, including \(L_1/L_2\) minimization.

Recently, the variational properties of \(L_1/L_2\) minimization, including partly smoothness and prox-regularity, were rigorously established in \cite{10542092}. These findings demonstrated that proximal-friendly methods, such as ADMM\(_p\) \cite{Tao20} and ADMM\(_p^+\) \cite{TaoZhang23}, preserve the finite identification property \cite{LW04}, which enables the identification of low-dimensional manifolds within a finite number of iterations. A two-phase heuristic acceleration framework, combining ADMM\(_p\)/ADMM\(_p^+\) with a globalized semismooth Newton method, was developed in \cite{10542092}, achieving high-order convergence rates while significantly reducing computational complexity. However, these methods typically converge to d-stationary \cite{PangRazAlv,TaoZhang23} or critical points rather than global solutions.

For the analysis of global convergence to a d-stationary  or critical point in the SGPM \cite{ELX13, YEX14}, the accelerated schemes \cite{WYYL20}, ADMM\(_p\) \cite{Tao20}, ADMM\(_p^+\) \cite{TaoZhang23}, and the two-phase heuristic acceleration algorithm \cite{10542092}, it is usually assumed that the primal sequence \(\{\h x^k\}\) is bounded. However, this assumption is difficult to validate.

Moreover, \cite{ChenXuYe} established lower bounds on the absolute values of the nonzero entries of any local optimal solution of the unconstrained model with the \(\|{\h x}\|_p^p\)  \((0 < p < 1)\) regularizer. These bounds facilitate the identification of nonzero entries, raising the question of whether similar bounds can be derived for local minimizers of the unconstrained \(L_1/L_2\) model.

In terms of computational complexity, the constrained model with the \(\|{\h x}\|_p^p\) \((0 < p < 1)\) regularizer, has been shown to be strongly NP-hard \cite{Ge2011} for both \(\mathcal{X} = \mathbb{R}^n_+\) and \(\mathbb{R}^n\). The unconstrained model with the \(\|{\h x}\|_p^p\) regularizer, is also confirmed to be strongly NP-hard \cite{Chen14}. Additionally, \cite{HuoChen} proved hardness results for a relaxed family of penalty functions involving the \(L_2\) loss, where the regularization terms include \(L_0\) hard-thresholding \cite{AF01} and SCAD \cite{FL01}. However, none of these results directly address whether finding the global minimum of the constrained and unconstrained $L_1/L_2$ models is strongly NP-hard.

\subsection{This paper}
We focus on the following constrained $L_1/L_2$ model:

\begin{eqnarray}\label{L1o2Con}
\begin{array}{ll}
\min\limits_{{\h x}}& \displaystyle{\frac{\|\mathbf{x}\|_{1}}{\|\mathbf{x}\|_{2}}} \\[0.2cm]
\text{subject to} & \mathbf{x} \in \mathcal{H}:=\{{\h x} \in {\cal X} \mid A {\h x} = {\h b}\},
\end{array}
\end{eqnarray}
and the unconstrained one:

\begin{eqnarray}\label{L1o2uncon}
\min_{{\h x}\in {\cal X}} \gamma \frac{\|{\h x}\|_{1}}{\|{\h x}\|_{2}}+\frac{1}{2}\|A {\h x}-\mathbf{b}\|_{2}^{2},
\end{eqnarray}
where $A \in \mathbb{R}^{m \times n}$ ($m \ll n$) is a given sensing matrix, ${\h b} \in \mathbb{R}^n$ (${\h b} \neq \mathbf{0}$), and $\gamma > 0$ is a balancing parameter.

First, we derive a uniform upper bound, in terms of $L_2$-norm, for any local minimizer of both the constrained model (\ref{L1o2Con}) and the unconstrained model (\ref{L1o2uncon}). This result shows that the boundedness typically required for global convergence analysis of first-order algorithms, such as those in \cite{ELX13, YEX14, Tao20,TaoZhang23,10542092,WYYL20}, for solving these models can be imposed with \emph{a priori}.

Second, we establish either a lower or upper bound for the nonzero entries in any local minimizer of the unconstrained model (\ref{L1o2uncon}). These theoretical results sheds light on the following two phenomena:

\begin{enumerate}
   \item As noted in \cite{WYYL20}, when the dynamic range of the ground-truth signal is larger, $L_1/L_2$ minimization yields a smaller relative error. The established bounds in this paper can help explain this phenomenon.

\item In our previous work \cite{10542092}, we introduced a hard-shrinkage step within the heuristic acceleration framework on active manifolds (see \cite[Algorithm 1]{10542092}). This step, applied between Phase I and Phase II, significantly improved recovery accuracy when processing real data. Further discussions can be found in \cite[Section V.C]{10542092}.

\end{enumerate}

Last but not least, to the best of our knowledge, the computational complexity of $L_1/L_2$ minimization remains unexploited.
 This raises the following fundamental open questions:
\begin{itemize}
\item ``{\it Whether the constrained model (\ref{L1o2Con}) and the unconstrained model (\ref{L1o2uncon}) are NP-hard?}"
\item
``{\it Whether the constrained model (\ref{L1o2Con}) and the unconstrained model (\ref{L1o2uncon}) are strongly NP-hard?}"
\end{itemize}
We answer both questions affirmatively.
In light of the growing interest in ratio minimization problems \cite{ZHOU2021, ZHOU2019247, XU2021486}, we also extend our analysis to the complexity of $L_p$ ($0 < p \leq 1$) over $L_q$ ($1 < q < +\infty$) minimization problem.

\subsection{Our contributions}
In summary,  the novelties in this paper are three-fold:
\begin{itemize}
\item We provide unified upper bounds in terms of the $L_2$-norm for any local minimizer of (\ref{L1o2Con}) and (\ref{L1o2uncon}),
 respectively.
 Simultaneously, we derive the upper and lower bounds for the nonzero entries in any local minimizer of (\ref{L1o2uncon});

\item We prove that both (\ref{L1o2Con}) and (\ref{L1o2uncon}) are strongly NP-hard;

\item We prove that minimization of the $L_p$ norm ($0 < p \leq 1$) over the $L_q$ norm ($1 < q < +\infty$) is strongly NP-hard.

\end{itemize}
\medskip
\subsection{Organization} We structure the rest of this paper as follows:  Section \ref{Sec-Preliminaries} presents notations and several useful lemmas. In Section \ref{wellDef}, we elaborate the bound theory for $L_1/L_2$ minimization. In Section \ref{NPHard}, we prove that both (\ref{L1o2Con}) and (\ref{L1o2uncon}) are strongly NP-hard. In Section \ref{Lpq}, we point out that \(L_p\) (\(0 < p \leq 1\)) over \(L_q\) (\(1 < q < +\infty\)) minimization is also strongly NP-hard.
 Discussions and conclusions are given in Sections \ref{dis}.

%\section{Related works}\label{RelatedW}

\section{Preliminary}\label{Sec-Preliminaries}
\subsection{Notations} \label{Notations}
Let bold letters denote vectors, e.g., ${\bm x} \in \mathbb{R}^{n}$, and let $x_{i}$ and $|\bm{x}|$ denote the $i$-th entry and the absolute value of the vector, respectively. Given two vectors ${\bm x}, {\bm y} \in \mathbb{R}^{n}$, $\langle {\bm x}, {\bm y} \rangle$ denotes their standard inner product: $\sum_{i=1}^n x_i y_i$. The notation $\|\bm{x}\|_p$ refers to the $p$-norm, defined as
\[
\|\bm{x}\|_p = \left(\sum_{i=1}^n |x_i|^p \right)^{1/p}
\]
for $0 < p < \infty$. The subscript $p$ in $\|\cdot\|_p$ is omitted when $p=2$. Define $[n]:=\{1, 2, \ldots, n\}$. The symbol $\sharp(\mathcal{D})$ denotes the cardinality of $\mathcal{D}$. $\sign(\h x)$ is defined as a vector of the same length as $\h x$, with its $i$-th component equal to zero if $x_i = 0$, and otherwise equal to the sign of each component of ${\h x}$. ${\text{supp}}({\bm x}):=\{i \in [n] \ | \ x_i \neq 0\}$ is the support of ${\bm x}$. The notation $\otimes$ represents a Kronecker product. Given two matrices $A_r = (a^{(r)}_{i,j}) \in \mathbb{R}^{m \times n}$ and $A_c \in \mathbb{R}^{s \times t}$, $A_r \otimes A_c = (a^{(r)}_{i,j} A_c)$. The notation ${\text{vec}}(\cdot)$ refers to the stacking of columns. Let $\mathbb{R}^n_+$ and $\mathbb{R}^n_-$ denote the sets of nonnegative and nonpositive vectors, respectively, and $\mathbb{R}^n_+ / \mathbb{R}^n$ refers to either $\mathbb{R}^n_+$ or $\mathbb{R}^n$. For a matrix $A \in \mathbb{R}_+^{m \times n}$, this means each entry of $A$ is nonnegative. The vector ${{\bm e}}^{(n)}$ represents a vector with all entries equal to $1$ in $\mathbb{R}^n$, while ${\bm e}_i$ denotes the vector with $1$ in the $i$-th entry and $0$ elsewhere. We omit the superscript ${{\bm e}}^{(n)}$ when there is no ambiguity. The symbol $I_{n}$ denotes the $n \times n$ identity matrix. For a set of vectors ${\cal U} = \{{\bm v}_1, \dots, {\bm v}_r\}$, ${\text{span}}({\cal U})$ represents the linear space spanned by ${\cal U}$. Given a matrix $A \in \mathbb{R}^{m \times n}$ or a vector $\bm{x} \in \mathbb{R}^{n}$ and an index set $\Lambda \subseteq [n]$, we use $A_{\Lambda}$ and $\bm{x}_{\Lambda}$ to denote $A[:, i]_{i \in \Lambda}$ and $\bm{x}[i]_{i \in \Lambda}$, respectively. For a square matrix $A$, the notation $A \succ {\bf 0}$ means that $A$ is positive definite. The symbols $\mathcal{L}_{c}^*$ and $\mathcal{L}_{u}^*$ denote the local minimizers of the constrained problem (\ref{L1o2Con}) and the unconstrained problem (\ref{L1o2uncon}), respectively. A $\mathcal{C}^2$ function $f$ is twice continuously differentiable. The set $\mathbb{R} \setminus \{0\}$ represents the real numbers excluding zero. Throughout the paper, we assume that $\bm{b} \neq 0$, $\mathcal{H} \neq \emptyset$, and ${\cal X} = \mathbb{R}^n_+ / \mathbb{R}^n$.

\subsection{Preparatory lemmas}

In the following, we present the definition of dynamical range of a signal.
\begin{Def}\label{dyrang}
    Given a signal ${\h x}\in \mathbb{R}^n_+ / \mathbb{R}^n$, let $\Lambda:={\text{supp}}({\h x})$,
    we define the dynamical range of ${\h x}$ as
    \begin{eqnarray*}R({\h x}):=\frac{\max_{i\in\Lambda}|x_i|}{\min_{i\in{\Lambda}}|x_i|}.\end{eqnarray*}
\end{Def}
\noindent Next, we provide several lemmas that will be used in later analysis.
\begin{lemma}\label{pqineq} (i) If \( p \in (0,1) \) and \( a, b \geq 0 \), then
\[
(a + b)^p \leq a^p + b^p,
\]
with equality holding if and only if either \( a \) or \( b \) is zero;\\
(ii) If \( q \in (1,+\infty) \) and \( a, b \geq 0 \), then
\[
(a + b)^q \ge a^q + b^q,
\]
with equality holding if and only if either \( a \) or \( b \) is zero.

\end{lemma}

\begin{proof}(i)
Clearly, the inequality holds with equality when either \( a \) or \( b \) is zero. We now assume both \( a \) and \( b \) are nonzero. Define the one-variable function \( f(t) := (1+t)^p - 1 - t^p \) for \( t > 0 \) with derivative \( f'(t) = p(1+t)^{p-1} - pt^{p-1} \). Since \( f'(t) < 0 \) for all \( t \in (0, +\infty) \), the function is strictly decreasing on this interval. Note that \( f(0) = 0 \), it follows that \( f(t) < 0 \) for \( t \in (0, +\infty) \).

Now, substituting \( t = a/b \) (assuming both \( a \) and \( b \) are nonzero), we get
\begin{eqnarray*}
\left(1 + \frac{a}{b}\right)^p - 1 - \left(\frac{a}{b}\right)^p < 0 \Leftrightarrow (a + b)^p < a^p + b^p,
\end{eqnarray*}
where the inequality follows directly from the fact that \( f(t) \) is strictly decreasing.
 The item (ii) can be shown similarly.
\end{proof}
From the above lemma, we immediately obtain the following result.
\begin{lemma}\label{pqineqnorm} (i) If \( p \in (0,1) \) and \( {\h a}, {\h b} \in \mathbb{R}^{n} \), then
\[
\| {\h a} + {\h b}\|^p_p  \leq \|{\h a}\|^p_p  + \|{\h b}\|^p_p,
\]
with equality holding if and only if  $\langle {\h a} ,{\h b}\rangle =0$;\\
(ii) If \( q \in (1,+\infty) \) and \( {\h a}, {\h b} \in \mathbb{R}^{n}\), then
\[
\| {\h a} + {\h b}\|^q_q  \ge\|{\h a}\|^q_q  + \|{\h b}\|^q_q,
\]
with equality holding if and only if $\langle {\h a} ,{\h b}\rangle =0$.

\end{lemma}

We now state and show the next lemma.

\begin{lemma}\label{Ray}
Let $\h x(\neq{\bf 0})$ and ${\h e}$ be two vectors in $ \mathbb{R}^{n}$ and ${\h x}\not\in{\text{span}}({\h e})$. Define the  matrix
$$G={\h x}{\h z}^\top +{\h e}{\h x}^\top$$ and
${\h z}={\h e}+\delta {\h x}$ where $\delta(\neq 0)$ is a constant.
Then,
the matrix $G$ has two nonzero eigenvalues:
$$\lambda_{1,2}=\frac{(2{\h e}^\top {\h x}+\delta \|\h x\|^2) \pm \sqrt{\delta^2\|\h x\|^4+4(n+\delta{\h e}^\top{\h x})\|\h x\|^2}}{2},$$
and   eigenvalue $0$ with the multiplicity of $n-2$.

\end{lemma}

\begin{proof}
Suppose that ${\h y}\in[{\text{span}}({\h x},{\h e})]^\bot$.
Then,
$$ \langle {\h y}, {\h x }({\h e}+\delta{\h x})^\top +{\h e}{\h x}^\top \rangle =0.$$
It implies that ${\h y}$ is also an eigenvector of $G$ associated with $0$.
On the contrary, suppose that $\h y$ be an eigenvector of $G$ associated with $0$,
it leads to
$${\h y}^\top [ {\h x }({\h e}+\delta{\h x})^\top +{\h e}{\h x}^\top]=0.$$
By using the linear  independence of ${\h x}$ and ${\h e}$,
it further implies that
 ${\h e}^\top{\h y} =0$ and $ {\h x}^\top{\h y}=0$. Thus,
${\h y}\in[{\text{span}}({\h x},{\h e})]^\bot$.
Then, $[{\text{span}}({\h x},{\h e})]^\bot$ is the eigenvectors space
 associated $0$.
${\text{dim}}( [{\text{span}}({\h x},{\h e})]^\bot)=n-2$. The multiplicity of $0$ is $n-2$.

Let $\lambda$ be any nonzero eigenvalue of $G$ and $\h y$  be an eigenvector of $G$ associated with $\lambda$, i.e.,
\begin{eqnarray*} ({\h x}{\h z}^\top+{\h e}{\h x}^\top){\h y}=\lambda {\h y}. \end{eqnarray*}
Thus, ${\h y}=\theta_1 {\h x}+\theta_2 {\h e}$.  By substituting it into the above equation, we have that
$$  ({\h x}{\h z}^\top+{\h e}{\h x}^\top)(\theta_1 {\h x}+\theta_2 {\h e})=\lambda(\theta_1 {\h x}+\theta_2 {\h e}).$$
By comparing the coefficients of ${\h x}$ and ${\h e}$ on both sides of the above equation, it leads to
\begin{eqnarray*} \left(\begin{array}{cc} {\h z}^\top {\h x} & {\h z}^\top {\h e}\\ {\h x}^\top {\h x}&{\h e}^\top {\h x}\end{array}\right)\left(\begin{array}{c}\theta_1 \\ \theta_2 \end{array}\right)=\lambda\left(\begin{array}{c}\theta_1\\ \theta_2\end{array} \right). \end{eqnarray*}
It implies that $\lambda$ is also the eigenvalues of the left $2\times 2$ matrix.
With some elementary calculations, we have two nonzero eigenvalues are
\begin{eqnarray*} \lambda_{1,2}=\frac{({\h z}^\top {\h x}+{\h e}^\top {\h x})\pm \sqrt{({\h z}^\top {\h x}-{\h e}^\top {\h x})^2+4({\h z}^\top {\h e})({\h x}^\top {\h x})}}{2}. \end{eqnarray*}
By invoking ${\h z}={\h e}+\delta {\h x}$, the assertion follows directly.
\end{proof}

For the convenience of later analysis, we provide the gradient and Hessian matrix of a ratio function as follows. Let \( q: \mathbb{R}^s / \mathbb{R}^s_+ \to \overline{\mathbb{R}} := \mathbb{R} \cup \{+\infty\} \) and \( p: \mathbb{R}^s / \mathbb{R}^s_+ \to \mathbb{R} \setminus \{0\} \), where both \( q(\cdot) \) and \( p(\cdot) \in \mathcal{C}^2 \). We consider the ratio of \( q \) and \( p \) over \( \mathbb{R}^s / \mathbb{R}^s_+ \):

\[
R(\mathbf{u}) = \frac{q(\mathbf{u})}{p(\mathbf{u})}.
\]
The gradient of \( R(\cdot) \) is given by

\begin{align*}
    \nabla R(\mathbf{u}) &= \frac{p(\mathbf{u}) \nabla q(\mathbf{u}) - q(\mathbf{u}) \nabla p(\mathbf{u})}{p^2(\mathbf{u})} \nn\\
    &= \frac{\nabla q(\mathbf{u}) - R(\mathbf{u}) \nabla p(\mathbf{u})}{p(\mathbf{u})}.
\end{align*}
Differentiating again yields the Hessian of \( R(\cdot) \) as follows:
\begin{align}\label{HessR}
\nabla^2 R(\mathbf{u}) = & \frac{\nabla^2 q(\mathbf{u})}{p(\mathbf{u})} - \frac{1}{p^2(\mathbf{u})} \left( \nabla p(\mathbf{u}) \nabla q(\mathbf{u})^\top\right.\nn\\
 &  \left.+\nabla q(\mathbf{u}) \nabla p(\mathbf{u})^\top+  \nabla^2 p(\mathbf{u}) q(\mathbf{u}) \right) \nn\\
&+ 2 \frac{q(\mathbf{u})}{p^3(\mathbf{u})} \nabla p(\mathbf{u}) \nabla p(\mathbf{u})^\top.
\end{align}
Finally, we review an important property of Kronecker products: given three matrices $A_c\in\mathbb R^{m\times n}$, $A_r\in\mathbb R^{p\times s}$ and $\h X\in{\mathbb R}^{n\times s}$, it holds that
\begin{eqnarray*}(A_r\otimes A_c){\text{vec}}(\h X)={\text{vec}}(A_c {\h X} A_r^\top).\end{eqnarray*}

\section{Bound theory for  $L_1/L_2$ minimization} \label{wellDef}
First, we  provide  unified upper bounds in terms of $L_2$-norm for any local minimizer of (\ref{L1o2Con}) and (\ref{L1o2uncon}), respectively.
Second, we derive the upper/lower bound of nonzero entries of any local minimizer of
 (\ref{L1o2uncon}).

%Thus, each diagonal element of $(A_{\Lambda})^\top A_{\Lambda}$  is positive.
\subsection{Uniform upper bound for any local minimizer}
Below, we present the results on the $L_2$-norm upper bounds for any local minimizer of the model (\ref{L1o2Con}).
%Obviously, $(A_{\Lambda})^\top A_{\Lambda}\succ 0$ ($\Lambda:={\text{supp}}(\h x)$)  since ${\h x}\in{\cal L}_{u}^*/{\cal L}_{c}^*$ \cite[Theorem 5]{TaoZhang23}/\cite[Lemma 2]{TaoZhang23}.

\begin{theorem} \label{LocmiCon} Consider the constrained model (\ref{L1o2Con}).
For any ${\h x}\in{\cal L}_{c}^*$,  $\|\h x\|\le \frac{\|\h b\|}{\sqrt{\sigma}}$ where where $\sigma$ denotes the smallest nonzero eigenvalue of $A^\top A$.
\end{theorem}

\begin{proof} Let ${\h x}\in{\cal L}_{c}^*$ and  $\Lambda ={\text{supp}}(\h x)$, then $(A_{\Lambda})^\top A_{\Lambda}\succ 0$ due to
\cite[Lemma 2]{TaoZhang23}. Thus
\begin{eqnarray*}\sqrt{\sigma_{\min}(A_{\Lambda}^\top A_{\Lambda})}\|{\h x}_{\Lambda}\|\le\|A_{\Lambda}{\h x}_{\Lambda}\|=\|{\h b}\|, \end{eqnarray*}
where  $\sigma_{\min}(\cdot)$ denotes the minimum nonzero eigenvalue.
Consequently, $\|\h x\|\le \frac{\|\h b\|}{\sqrt{\sigma}}$
 because $\sigma_{\min}(A_{\Lambda}^\top A_{\Lambda})\ge \sigma$.
\end{proof}
\medskip
Next, we present two different approaches to derive the $L_2$-norm upper bounds for any local minimizer of  (\ref{L1o2uncon}).

\begin{theorem}\label{Locmi}Consider the unconstrained model (\ref{L1o2uncon}).
\begin{itemize}
\item[(i)] For any ${\h x}\in{\cal L}_{u}^*$ and $\Lambda ={\text{supp}}(\h x)$, the following holds:
 \begin{eqnarray*}\|{\h x}\|\le\displaystyle{ \frac{\sigma^{-1}\|A^\top {\h b}\|+\sqrt{(\sigma^{-1}\|A^\top {\h b}\|)^2+4\sigma^{-1}\gamma\sqrt{n}}}{2},}\end{eqnarray*}
    where $\sigma$ denotes the smallest nonzero eigenvalue of $A^\top A$;
    \item[(ii)] For any ${\h x}\in{\cal L}_{u}^*$,  $\|\h x\|\le \frac{\|\h b\|}{\sqrt{\sigma}}$ where $\sigma$ denotes the smallest nonzero eigenvalue of $A^\top A$.
   \end{itemize}
    \end{theorem}

    \begin{proof} (i) Let ${\h x}\in{\cal L}_{u}^*$. Denote $a=\|{\h x}\|_1$ and $r=\|\h x\|_2$.
We denote ${\Lambda}={\text{supp}}({\h x})$ and $s=\sharp(\Lambda)$.
We can show that ${\h y}={\h x}_{\Lambda}$ is also a local minimizer of
\begin{eqnarray}\label{restopt}\min_{{\h y}\in {\cal X}^s} { f}(\h y)=\gamma \frac{\|{\h y}\|_{1}}{\|{\h y}\|_{2}}+\frac{1}{2}\|A_{\Lambda}{\h y}-{\h b}\|^2,\end{eqnarray}
where ${\cal X}^s={\mathbb R}^s/{\mathbb R}^s_+$.
From the first-order optimality condition, one has
\begin{eqnarray*}\gamma\left(\frac{{ \sign(\h y)}}{r}-\frac{a}{r^3}{\h y}\right)+(A_{\Lambda})^\top(A_{\Lambda}{\h y}-{\h b})=0, \end{eqnarray*}
both for ${\cal X}^s={\mathbb R}^s$ and ${\cal X}^s={\mathbb R}^s_+$.
Thus
\begin{eqnarray*} {\h y}= ((A_{\Lambda})^\top A_{\Lambda})^{-1}\left[ (A_{\Lambda})^\top{\h b}-\gamma\left(\frac{\sign(\h y)}{r}-\frac{a}{r^3}{\h y}\right)\right]. \end{eqnarray*}
Taking $L_2$-norm on both sides of the above inequality and then using Cauchy-Schwarz inequality,
we have
\begin{eqnarray}\label{rupbound1} r\le{\sigma_{\min}(A_{\Lambda}^\top A_{\Lambda})}^{-1}\left[\| (A_{\Lambda})^\top{\h b}\|+{\gamma}\left\|\frac{\sign(\h y)}{r}-\frac{a}{r^3}{\h y} \right\| \right]. \end{eqnarray}
%where ${\tilde \sigma}$ denotes the smallest nonzero eigenvalue of $(A_{\Lambda})^\top A_{\Lambda}$ which is positive.
Next, we calculate the last term in the above inequality,
\begin{eqnarray*}&&\left\|\frac{\sign(\h y)}{r}-\frac{a}{r^3}{\h y} \right\| =\frac{1}{r^3}\|{\bf e}r^2-a|\h y| \|\nn\\
                        &&=\frac{1}{r^3}\sqrt{r^4 s - r^2 a^2}\le \frac{\sqrt{s}}{r}.\end{eqnarray*}
Note that
$\|(A_{\Lambda})^\top {\h b}\|\le \|A^\top {\h b}\|$.
In addition, we have that $\sigma_{\min}(A_{\Lambda}^\top A_{\Lambda})\ge \sigma$ since the columns of $A_{\Lambda}$ is a subset of
the columns of $A$. By substituting these three inequalities into (\ref{rupbound1}),  we have that
\begin{eqnarray*} r\le{\sigma}^{-1}\left[\| A^\top{\h b}\|+\frac{\gamma}{r}\sqrt{n} \right]. \end{eqnarray*}
It amounts to solving a quadratic inequality  of a single variable $r$, i.e.,
$$ r^2-\sigma^{-1}\|A^\top {\h b}\| r-\sigma^{-1}\gamma\sqrt{n}\le 0.$$
Thus the assertion (i) follows directly.

(ii) For both cases of ${\cal X}={\mathbb R}^n/{\mathbb R}^n_+$, it follows from the first-order optimality condition that \begin{eqnarray*} \gamma\left(\frac{{\text{sign}}({\h x}_{\Lambda})}{r}-\frac{a}{r^3}{\h x}_{\Lambda}\right)+A_{\Lambda}^\top(A_{\Lambda}{\h x}_{\Lambda}-{\h b})=0. \end{eqnarray*}
By multiplying the above equation with ${\h x}_{\Lambda}^\top$ and
using ${\h x}_{\Lambda}^\top(\frac{{\text{sign}}({\h x}_{\Lambda})}{r}-\frac{a}{r^3}{\h x}_{\Lambda})=0$, we see that
\begin{eqnarray*} {\h x}_{\Lambda}^\top (A_{\Lambda}^\top A_{\Lambda}{\h x}_{\Lambda}-A_{\Lambda}^\top{\h b})=0\Rightarrow \|A_{\Lambda}{\h x}_{\Lambda}\|^2\le \|A_{\Lambda}{\h x}_{\Lambda}\|\|{\h b}\|. \end{eqnarray*}
Consequently, $\|\h x\|\le \frac{\|\h b\|}{\sqrt{\sigma}}$
 because $\sigma_{\min}(A_{\Lambda}^\top A_{\Lambda})\ge \sigma$.
\end{proof}

\begin{remark}
Theorems \ref{LocmiCon} and \ref{Locmi} present respectively a closed ball that contains all local minimizers for problems (\ref{L1o2Con}) and (\ref{L1o2uncon}). These two theorems indicate that the boundedness assumption for establishing the global convergence of the primal sequence generated by  the SGPM \cite{ELX13, YEX14}, the accelerated schemes in \cite{WYYL20}, ADMM$_p$ \cite{Tao20}, ADMM$_p^+$ \cite{TaoZhang23}, and the two-phase heuristic acceleration algorithm in \cite{10542092} is reasonable. These bounds can be computed in advance and imposed beforehand.
\end{remark}

\subsection{Upper/lower bounds of nonzero entries of local minimizers}
We now present the upper/lower bounds for the nonzero entries of the local solution to the model (\ref{L1o2uncon}).
    \begin{theorem}
    \label{bounduncon}
    Consider the unconstrained model (\ref{L1o2uncon}) with ${\cal X}={\mathbb R}^n/{\mathbb R}^n_+$.
For any ${\h x}\in{\cal L}_{u}^*$, let $a=\|\h x\|_1$, $r=\|\h x\|_2$ and $\delta_{i}={\h e}_i^\top (A_{\Lambda}^\top A_{\Lambda}) {\h e}_i$ $(i\in\Lambda)$  with ${\tilde\delta_i} =\delta_i/\gamma$.
Define
\begin{eqnarray*}\displaystyle{\Delta_i := \frac{4}{r^6}+\frac{12a^2}{r^8}-\frac{12 a{\tilde \delta_i}}{r^5}} \end{eqnarray*}
and
\begin{eqnarray*}&&\displaystyle{\kappa_{1,i}:=\frac{\frac{2}{r^3}-\sqrt{\frac{4}{r^6}+\frac{12a^2}{r^8}-\frac{12 a{\tilde \delta_i}}{r^5}}}{\frac{6a}{r^5}}}, \nn\\[0.2cm]
&&\displaystyle{\kappa_{2,i}:=\frac{\frac{2}{r^3}+\sqrt{\frac{4}{r^6}+\frac{12a^2}{r^8}-\frac{12 a{\tilde \delta_i}}{r^5}}}{\frac{6a}{r^5}}}.\end{eqnarray*}
\begin{itemize}
\item[(i)]  For each $i\in[n]$, if ${\tilde \delta_i} <\frac{a}{r^3}+\frac{1}{3ar}$, then $\Delta_i>0$;
\item[(ii)] For each $i\in[n]$, if $\frac{a}{r^3}\ge{\tilde \delta}_i$, then $|x_i|\ge {\kappa_{2,i}}$;
\item[(iii)]For each $i\in[n]$, if $\frac{a}{r^3}<{\tilde \delta_i} <\frac{a}{r^3}+\frac{1}{3ar}$,
then
$|x_i|\le \kappa_{1,i}$ or $|x_i|\ge \kappa_{2,i}$.
\end{itemize}
\end{theorem}

\begin{proof}
The following proof applies to both ${\cal X}^s = {\mathbb R}^s$ and ${\cal X}^s = {\mathbb R}^s_+$. Therefore, we present it in a unified manner.
Using the formula  (\ref{HessR}), for ${\h x} \in {\cal L}_u^*$ the generalized Hessian of (\ref{restopt}) at  ${\h y} = {\h x}_\Lambda$ is given as
\begin{align}\label{LocCond}
\nabla^2 f({\h y}) = & [\nabla^2 F({\h x})]_{\Lambda, \Lambda} \nn\\
= & (A_\Lambda)^\top A_\Lambda - \gamma \frac{1}{r^2} \left( \frac{{\h y} (\text{sign}({\h y}))^\top + \text{sign}({\h y}) {\h y}^\top}{r} \right)\nn\\
&  + \gamma \frac{3a {\h y} {\h y}^\top}{r^5} - \gamma \frac{a}{r^3} \succeq \mathbf{0}.
\end{align}
Consequently,
\begin{align*}
\nabla^2 f({\h y})
= & (A_\Lambda)^\top A_\Lambda - \gamma \frac{1}{r^2} \left( \frac{|{\h y}| {\h e}^\top + {\h e} |{\h y}|^\top}{r} \right)\nn\\
 &~ + \gamma \frac{3a |{\h y}| |{\h y}|^\top}{r^5} - \gamma \frac{a}{r^3}.
\end{align*}
Next, we compute all of the eigenvalues of $$Q = \frac{\gamma}{r^3}\left[ \left( |{\h y}| {\h e}^\top + {\h e} |{\h y}|^\top \right) - \frac{3a |{\h y}| |{\h y}|^\top}{r^2} + aI\right ].$$ Define $${\tilde Q} = |{\h y}| {\h e}^\top + ({\h e} + \delta |{\h y}|) |{\h y}|^\top$$ with $\delta = -\frac{3a}{r^2}$. Invoking Lemma \ref{Ray}, with ${\h e}^\top |{\h y}| = a$ and $\|{\h y}\|^2 = r^2$, the matrix ${\tilde Q}$ has two nonzero eigenvalues:
\[
\lambda_{1,2} = \frac{(2a + \delta r^2) \pm \sqrt{\delta^2 r^4 + 4sr^2 + 4r^2 \delta a}}{2}.
\]
In addition, $0$ is the eigenvalue of ${\tilde Q}$ with multiplicity $n-2$.

By substituting $\delta=-\frac{3a}{r^2}$ into the above equations, we further have
 $$\lambda_{1}=\frac{-a- \sqrt{4s r^2 -3a^2} }{2},\;\mbox{and}\; \lambda_2=\frac{-a+ \sqrt{4s r^2 -3a^2} }{2}.$$
  Thus the matrix ${\tilde Q}$ has two nonzero eigenvalues $\lambda_1,\lambda_2$ and $0$ with  the multiplicity of $(n-2)$.
  Nevertheless, the matrix $Q$ has two nonzero eigenvalues $\frac{\gamma}{r^3}(\lambda_1+a)<0$ and $\frac{\gamma}{r^3}(\lambda_2+a)>0$
  and $\frac{\gamma a}{r^3}$ with  the multiplicity of $(n-2)$.
  We define $\delta_{i}$  as ${\h e}_i^\top (A_{\Lambda})^\top (A_{\Lambda}) {\h e}_i$ $(i\in\Lambda)$ which are all positive.
  On the other hand, it follows from (\ref{LocCond}) and ${\h e}_i^{\top} \nabla^2 f({\h y}){\h e}_i\ge 0$ that
 \begin{eqnarray*}&&{\h e}_i^\top(A_{\Lambda})^\top (A_{\Lambda}){\h e}_i \nn\\
 &&\ge{\h e}_i^\top\left[\gamma\frac{1}{r^2}\left( \frac{|{\h y}|{\h e}^\top +{\h e}|{\h y|}^\top}{r}\right)-\gamma\frac{3a{|\h y|}{|\h y|}^\top}{r^5}+\gamma\frac{a}{r^3}\right] {\h e}_i. \end{eqnarray*}
  In addition, note that $x_i=y_i$ for $i\in\Lambda$, it further leads to
 \begin{eqnarray*}\tilde{\delta_i}\ge \frac{2|x_i|}{r^3}+\frac{a}{r^3}-\frac{3a x_i^2}{r^5},\;\;i\in\Lambda.
 \end{eqnarray*}
 Consider the unary quadratic inequality:
 \begin{eqnarray*}\label{tunar} \frac{3a}{r^5} t^2-\frac{2}{r^3} t+{\tilde \delta_i}-\frac{a}{r^3}\ge 0\end{eqnarray*}
  with respect to $t$.
 If $${\tilde \delta_i}< \frac{a}{r^3}+\frac{1}{3ar},$$
 then  $\displaystyle{\Delta_i=\frac{4}{r^6}+\frac{12a^2}{r^8}-\frac{12 a{\tilde \delta_i}}{r^5}>0}$.
 Thus, (i) is valid.
 Then $|x_i|\le \kappa_{1,i}$ or $|x_i|\ge \kappa_{2,i}$.

 For (ii)
 $\kappa_{1,i}<0$ due to $\frac{a}{r^3}\ge{\tilde \delta}_i$.
 Thus, $|x_i|\ge \kappa_{2,i}$.
 (iii) holds trivially.
\end{proof}

Theorem \ref{bounduncon} concerns the local minimizers of the  model (\ref{L1o2uncon}).
We now give an example to illustrate the bound theory in Theorem \ref{bounduncon}.
\begin{example}\label{boundex} Consider the  $L_1/L_2$ minimization problem for any given $\gamma>3/4$:
\begin{eqnarray*}\min f({\h x}):=\gamma\frac{\|{\h x}\|_1}{\|{\h x}\|_2}+\frac{1}{2}(x_1+x_2-1)^2. \end{eqnarray*}
Note that $A=(1,1)$ and $b=1$. We easily see that  there are three stationary points for this problem:
\begin{eqnarray*}{\h x}^{(1)}=(1,0)^\top,\; {\h x}^{(2)}=(0,1)^\top, \;{\h x}^{(3)}=(1/2,1/2)^\top. \end{eqnarray*}
Obviously, both ${\h x}^{(1)}$ and $ {\h x}^{(2)}$ are  global minimizers, and ${\h x}^{(3)}$ is
not a local minimizer.
We use the vector ${\h x}^{(1)}$ to explain the bound theory.
The support set of  ${\h x}^{(1)}$ is $\Lambda_1=\{1\}$ and hence $\delta_1={\h e}_1^\top (A_{\Lambda_1}^\top A_{\Lambda_1}) {\h e}_1=1$. For this minimizer, one has
$r=1$ and $a=1$. Thus,
$\Delta_1=16-12/\gamma$, and $\Delta_1>0$ when $\gamma>3/4$.
 We have the following two cases regarding the value of $\gamma$:
\begin{itemize}
\item Case (a): $3/4<\gamma<1$.\\
From Item (iii) of Theorem \ref{bounduncon},
\begin{eqnarray*}|x^{(1)}_1|\le \frac{2-\sqrt{16-12/\gamma}}{6},\;
\mbox{or}\;|x^{(1)}_1|\ge \frac{2+\sqrt{16-12/\gamma}}{6}.\end{eqnarray*}
\item Case (b): $\gamma\ge 1$.\\
From Item (ii) of Theorem \ref{bounduncon},
\begin{eqnarray*}|x^{(1)}_1|\ge \frac{2+\sqrt{16-12/\gamma}}{6}.\end{eqnarray*}
If $\gamma=1$, we have $|x^{(1)}_1|\ge 2/3$.
If $\gamma\to +\infty$, the lower bound of $\frac{2+\sqrt{16-12/\gamma}}{6}\to 1$.
It implies that the lower bound becomes tighter and tighter as $\gamma\to +\infty$.
\end{itemize}

\end{example}

%\begin{remark}
%%For $\Phi=\frac{1}{2}\|A{\h x}-{\h b}\|^2$, we also have a unified upper bound for the first-order stationary points.
%
%\end{remark}
%\begin{theorem}\label{Genmodel} Consider the unconstrained model (\ref{L1o2uncon}) with $\Phi({\h x})=\frac{1}{2}\|{\h x}\|^2$ and ${\cal X}={\mathbb R}_+^n$.
%For any ${\h x}^*\in{\cal L}_{u}^*$, let $a=\|\h x\|_1$, $r=\|\h x\|_2$ and $\delta_{i}={\h e}_i^\top (A_{\Lambda}^\top A_{\Lambda}) {\h e}_i$ $(i\in\Lambda)$  with ${\tilde\delta_i} =\delta_i/\gamma$.
%Assume that $\gamma>\frac{(\min_{i\in\Lambda}{\delta_i}){3ar^3}}{3a^2+r^2}$.
%If ${\tilde \delta_i}< \frac{a}{r^3}+\frac{1}{3ar}$,
%then define
%\begin{eqnarray*}\displaystyle{\kappa_{1,i}=\frac{\frac{2}{r^3}-\sqrt{\frac{4}{r^6}+\frac{12a^2}{r^8}-\frac{12 a{\tilde \delta_i}}{r^5}}}{\frac{6a}{r^5}}}, \; \;\;\;\displaystyle{\kappa_{2,i}=\frac{\frac{2}{r^3}+\sqrt{\frac{4}{r^6}+\frac{12a^2}{r^8}-\frac{12 a{\tilde \delta_i}}{r^5}}}{\frac{6a}{r^5}}}.\end{eqnarray*}
%\begin{itemize}
%\item[(i)]If $\frac{a}{r^3}\ge{\tilde \delta}_i$, then $x_i\ge {\kappa_{2,i}}$;
%\item[(ii)]If $\frac{a}{r^3}<{\tilde \delta_i} <\frac{a}{r^3}+\frac{1}{3ar}$,
%then
%$0<x_i\le \kappa_{1,i}$ or $x_i\ge \kappa_{2,i}$.
%%\item[(iii)]
%\end{itemize}
%\end{theorem}
%\begin{proof} The proof is similar to Theorem \ref{bounduncon}, thus we omit here.\end{proof}

\begin{remark} As observed in \cite[Section VI]{WYYL20}, a higher dynamic range in the ground truth ${\h x}^*$ leads to better recovery results. The rationale behind this phenomenon is that there exists an index $i$ such that $|x_i| \leq \kappa_{1,i}$, provided that $\kappa_{1,i} > 0$, and another index $j$ such that $|x_j| \geq \kappa_{2,j}$. Consequently, the following inequality holds:
\begin{eqnarray*}
R({\h x}^*) \geq \frac{\frac{2}{r^3} + \sqrt{\frac{4}{r^6} + \frac{12a^2}{r^8} - \frac{12 a {\tilde \delta_j}}{r^5}}}{\frac{2}{r^3} - \sqrt{\frac{4}{r^6} + \frac{12a^2}{r^8} - \frac{12 a {\tilde \delta_i}}{r^5}}}.
\end{eqnarray*}
This result provides insight into the phenomenon reported in \cite{WYYL20}.
\end{remark}
\begin{remark}In our previous work \cite{10542092}, as discussed in \cite[Section V-B]{10542092}, we observed that the performance of the two-phase heuristic acceleration algorithm deteriorates when the Hard Shrinkage Step (i.e., ${\text{HARD}}({\h x},\tau)$ in \cite[Algorithm 1]{10542092}) is omitted. Furthermore, the step size $\tau>0$ in the Hard Shrinkage Step has a significant impact on the recovery results, as illustrated in Fig. 3 of \cite{10542092}. Selecting an appropriate value for $\tau$ can lead to a lower relative error. Additionally, Theorem \ref{bounduncon} explains for the effectiveness of the Hard Shrinkage Step in dealing with realistic datasets.

\end{remark}

\section{Hardness of $L_1$ over $L_2$ minimization} \label{NPHard} In this section, inspired by the seminal works \cite{Ge2011,Chen14,HuoChen},
 we prove that both (\ref{L1o2Con}) and (\ref{L1o2uncon}) are strongly NP-hard.

\subsection{Strong NP-hardness of the constrained model (\ref{L1o2Con})}
First, we establish the NP-hardness of (\ref{L1o2Con}) when ${\cal X}={\mathbb R}_+^n/{\mathbb R}^n$. To show the  NP-hardness, we employ a polynomial-time reduction from the well-known NP-complete partition problem \cite{HJ}.
The partition problem can be described as follows: Given a set $\cal S$ of rational numbers $\{a_1, a_2, \ldots, a_n\}$, can we partition $\cal S$ into two disjoint subsets, ${\cal S}_1$ and ${\cal S}_2$, such that the sums of the elements in each subset are equal?

\begin{theorem}\label{theo0}
The $L_1$ over $L_2$ minimization problem (\ref{L1o2Con}) with ${\cal X}={\mathbb R}_+^n$ or ${\cal X}={\mathbb R}^n$
is  NP-hard.
\end{theorem}
\begin{proof}First, we prove the NP-hardness of problem (\ref{L1o2Con})  with ${\cal X}={\mathbb R}^n_+$.
 An instance of the partition problem can be reduced to an instance of  (\ref{L1o2Con}) with ${\cal X}={\mathbb R}^n_+$  where the optimal value of $\sqrt{n}$ if and only if the former has an equitable bipartition. Define ${\h a}=(a_1,a_2,\ldots,a_n)^\top$.
Consider the following problem
 \begin{eqnarray}\label{ConsNPHard} \begin{array}{rl}&\min\limits_{\substack{{\h u}:=\left({\h x},{\h y}\right)\in {\mathbb R}_+^{2n}\\
 {\h x}\in {\mathbb R}_+^{n},\; {\h y}\in {\mathbb R}_+^{n}} } P({\h u}) =\displaystyle{\frac{\|{\h u}\|_1}{\|{\h u}\|_2}}\\[0.4cm]
 s.t. &\left(\begin{array}{cc} {\h a}^\top & -{\h a}^\top\\
                               I_n           &   I_n\\
                               \end{array}\right)\h u =\left(\begin{array}{c}{\bf 0}\\{\h e}\end{array}\right).
 \end{array}
 \end{eqnarray}
 which is an instance of (\ref{L1o2Con}) by setting
 \begin{eqnarray}\label{Ab}{\cal A}:=\left(\begin{array}{cc} {\h a}^\top & -{\h a}^\top\\
                               I_n           &   I_n\\
                               \end{array}\right),\; \mbox{and}\; {\textit{ b}}:=\left(\begin{array}{c}{\bf 0}\\{\h e}\end{array}\right).\end{eqnarray}
 By noting
  $\|{\h x}\|_2^2+\|{\h y}\|_2^2\le \|{\h x}+{\h y}\|_2^2$, we have
  \begin{eqnarray*}P({\h u}) =\displaystyle{\frac{\|{\h x}\|_1+\|{\h y}\|_1}{\sqrt{\|{\h x}\|_2^2+\|{\h y}\|_2^2}}}\ge\displaystyle{\frac{\|{\h x}+{\h y}\|_1}{\|{\h x}+{\h y}\|_2}}=\frac{\|{\h e}\|_1}{\|\h e\|_2}=\sqrt{n}. \end{eqnarray*}
 The problem (\ref{ConsNPHard}) achieves a value of $\sqrt{n}$ if and only if $x_i = 1$ and $y_i = 0$, or $x_i = 0$ and $y_i = 1$, which generates an evenly partitioned set of $\cal S$. Conversely, if the entries of $\cal S$ admit an equitable bipartition, then the problem (\ref{ConsNPHard}) achieves the global solution, i.e., $P({\h u}) = \sqrt{n}$.

Analogously, the same instance of the partition problem can also be
reduced to   (\ref{L1o2Con}) with ${\cal X}={\mathbb R}^n$:
\begin{eqnarray*} \begin{array}{rl}&\min\limits_{\substack{{\h u}:=\left({\h x},{\h y}\right)\in {\mathbb R}^{2n}\\
 {\h x}\in {\mathbb R}^{n},\; {\h y}\in {\mathbb R}^{n}} }  P({\h u}) =\displaystyle{\frac{\|{\h u}\|_1}{\|{\h u}\|_2}}\\[0.5cm]
 s.t. &\left(\begin{array}{cc} {\h a}^\top & -{\h a}^\top\\
                               I_n           &   I_n\\
                               \end{array}\right)\h u =\left(\begin{array}{c}{\bf 0}\\{\h e}\end{array}\right).
 \end{array}
 \end{eqnarray*}

With a similar argument, we can show that the global minimum value of $\sqrt{n}$ is achieved if and only if ${\cal S}$ has an equitable bipartition.  Therefore,  the problem (\ref{L1o2Con}) with ${\cal X} = {\mathbb R}^n$ is NP-hard.
\end{proof}
\medskip
Second, we prove the   strong NP-hardness of (\ref{L1o2Con}) when ${\cal X}={\mathbb R}_+^n/{\mathbb R}^n$ by utilizing
3-partition problem. We recall the strongly NP-hard 3-partition problem: Given a multiset $\mathcal{T}$ containing $n = 3m$ integers $\{a_1, a_2, \ldots, a_n\}$, where the sum of the elements in $\mathcal{T}$ equals $m\kappa$, and each integer $a_i$  satisfies $\kappa/4< a_i< \kappa/2$. The question is whether it is possible to partition set $\mathcal{T}$ into $m$ subsets, with each subset summing to $\kappa$. It implies that each subset must consist of  three elements.

\begin{theorem}\label{theo1}
The $L_1$ over $L_2$ minimization problem (\ref{L1o2Con}) with ${\cal X}={\mathbb R}_+^n$ or ${\cal X}={\mathbb R}^n$
is strongly NP-hard.
\end{theorem}

\begin{proof} First,
we describe a reduction from an instance of the 3-partition problem
to an instance of  (\ref{L1o2Con}) with ${\cal X}=\mathbb R_+^n$ where the optimal value $\sqrt{n}$ if and only if the former
has an equitable 3-partition.
Consider  the following minimization problem:
\begin{eqnarray}\label{3par}
&&\begin{array}{rll}
\min\limits_{{\h u}:={\text{vec}}({\h X})\in \mathbb R^{m n}_+}& P(\h u) =\displaystyle{\frac{\|\h u\|_1}{\|\h u\|_2}}&\\[0.3cm]
%s.t.& \sum_{j=1}^m x_{ij} =1,\ \forall\;i\in [n],&\\[0.2cm]
%    & \sum_{i=1}^n a_i x_{ij} =B,\ \forall\;j\in [m],&\\[0.2cm]
s.t. &\left( ({\h e}^{(m)})^\top \otimes I_n\right){\h u}={\h e}^{(n)},&\\[0.1cm]
     &\left({I}_m \otimes {\h a}^\top\right){\h u}=\kappa ({\h e}^{(m)})^\top,&\\[0.1cm]
  %  & x_{ij}\ge0,&\forall \; i\in[n], \ j\in[m],\\[0.2cm]
    & {\h x}_j=(x_{1j},\ldots,x_{nj})^\top\in \mathbb R^n_+,&\\[0.1cm]
    &\; \forall\; j\in[m],& \\[0.1cm]
    &{\h {X}} = ({\h {x}}_1, \dots, {\h {x}}_m) \in \mathbb{R}_+^{n \times m}.&
\end{array}\nn\\
\end{eqnarray}
The first equality constraint implies that
\[
{\h X}{\h e}^{(m)} = {\h e}^{(n)}.
\]
The second equality constraint is equivalent to
\[
{\h a}^\top {\h X} = \kappa ({\h e}^{(m)})^\top.
\]
Note that the objective function value of (\ref{3par}) satisfies
\[
P({\h u}) \geq \frac{\left\|\sum_{j=1}^m {\h x}_j\right\|_1}{\left\|\sum_{j=1}^m {\h x}_j\right\|_2} = \sqrt{n}.
\]
The optimal value is achieved when the matrix ${\h X} = ({\h x}_1, \dots, {\h x}_m) \in \mathbb{R}^{n \times m}$, and each row has exactly one nonzero entry equal to 1, while the sum of each column is 3. This generates an equitable 3-partition of the entries of ${\cal T}$. Conversely, if the entries of ${\cal T}$ have an equitable 3-partition, then the problem (\ref{3par}) must achieve the global solution, i.e.,
$P({\h u}) = \sqrt{n}.$

For the same instance of the 3-partition problem, we consider the following minimization problem in the form (\ref{L1o2Con}) with ${\cal X} = {\mathbb R}^n$.

\begin{eqnarray*}
&&\begin{array}{rll}
\min\limits_{{\h u}:={\text{vec}}({\h X})\in \mathbb R^{m n}}& P(\h u) =\displaystyle{\frac{\|\h u\|_1}{\|\h u\|_2}}&\\[0.3cm]
%s.t.& \sum_{j=1}^m x_{ij} =1,\ \forall\;i\in [n],&\\[0.2cm]
%    & \sum_{i=1}^n a_i x_{ij} =B,\ \forall\;j\in [m],&\\[0.2cm]
s.t. &\left( ({\h e}^{(m)})^\top \otimes I_n\right){\h u}={\h e}^{(n)},&\\[0.1cm]
     &\left({I}_m \otimes {\h a}^\top\right){\h u}=\kappa({\h e}^{(m)})^\top,&\\[0.1cm]
  %  & x_{ij}\ge0,&\forall \; i\in[n], \ j\in[m],\\[0.2cm]
    & {\h x}_j=(x_{1j},\ldots,x_{nj})^\top\in \mathbb R^n,\; \forall\; j\in[m],& \\[0.1cm]
    &{\h {X}} = ({\h {x}}_1, \dots, {\h {x}}_m) \in \mathbb{R}^{n \times m}.&
\end{array}
\end{eqnarray*}
The optimal value of the problem above is also $\sqrt{n}$. By a similar argument, we can demonstrate that the problem (\ref{L1o2Con}) with ${\cal X} = \mathbb{R}^n$ is also strongly NP-hard.
\end{proof}

\subsection{Strong NP-hardness of the unconstrained model (\ref{L1o2uncon})}
%Next, we show that (\ref{L1o2uncon}) is NP-hard in the following theorem.
First, we establish the NP-hardness of (\ref{L1o2uncon}) for the case where \(\mathcal{X} = \mathbb{R}_+^n / \mathbb{R}^n\). This is demonstrated by performing a polynomial-time reduction from the partition problem.
\begin{theorem}\label{theo2}
The $L_1$ over $L_2$ minimization problem (\ref{L1o2uncon}) with ${\cal X}={\mathbb R}_+^n$ or ${\cal X}={\mathbb R}^n$
is NP-hard.
\end{theorem}
\begin{proof}
We still use the partition problem.
% which can be described as:
%given  a set ${\cal S}$ with
%rational numbers $\{a_1,a_2,\ldots, a_n\}$ with number $2b$.
%Is there a way to partition ${\cal S}$ into two disjoint subsets
%${\cal S}_1$ and ${\cal S}_2$ such that
%the sum of the numbers in ${\cal S}_1$ and the sum of the numbers in ${\cal S}_2$
%both equal to $b$?
Recall that ${\h a}=(a_1,\ldots,a_n)^\top \in{\mathbb R}^n$.
We consider the minimization problem of the form:

\begin{eqnarray}\label{L1o2twocatnon}
\begin{array}{ll}
\min \limits_{\substack{{\h u}:=\left({\h x},{\h y}\right)\in {\mathbb R}_+^{2n}\\
 {\h x}\in {\mathbb R}_+^{n},\; {\h y}\in {\mathbb R}_+^{n}} } &P({\h u})=\displaystyle{\frac{1}{2}\frac{\|{\h u}\|_1}{\|{\h u}\|_2}}+\|{\cal A}{\h u}-{\textit{ b}}\|_2,
                   % &{\h x}\ge {\bf 0},\;{\h y}\ge {\bf 0}.
                    \end{array}\nn\\
\end{eqnarray}
where ${\cal A}$ and ${\textit{b}}$ are defined in (\ref{Ab}), the optimal value of (\ref{L1o2twocatnon}) is $\frac{\sqrt{n}}{2}$. This value is attained when $\|{\h a}^\top ({\h x} - {\h y})\|^2 = 0$, $\|{\h x} + {\h y} - {\h e}\| = 0$, and $\langle {\h x}, {\h y} \rangle = 0$. The global minimum is achieved if and only if $x_i = 1$ and $y_i = 0$, or $x_i = 0$ and $y_i = 1$, which generates an equitable bipartition of the entries in $\cal S$.

Conversely, if the entries of $\cal S$ admit an equitable bipartition, then the problem (\ref{L1o2twocatnon}) attains the global solution. Thus, the NP-hardness of the problem (\ref{L1o2uncon}) with ${\cal X} = {\mathbb{R}}_+^n$ is established. By a similar argument, we can prove the NP-hardness of the problem (\ref{L1o2uncon}) with ${\cal X} = \mathbb{R}^n$.
\end{proof}
%Next, we prove a stronger result.
Subsequently, we prove the strong NP-hardness of (\ref{L1o2uncon}) under the same condition \(\mathcal{X} = \mathbb{R}_+^n / \mathbb{R}^n\).

\begin{theorem}\label{theo3}
The $L_1$ over $L_2$ minimization problem (\ref{L1o2uncon}) with ${\cal X}={\mathbb R}_+^n$ or ${\cal X}={\mathbb R}^n$
is strongly NP-hard.
\end{theorem}

\begin{proof}
%We still use a polynomial time reduction from strongly NP-hard 3-partition problem \cite{HJ}. %Given a multiset $\mathcal{S}$ containing $n = 3m$ integers $\{a_1, a_2, \ldots, a_n\}$, where the sum of the elements in $\mathcal{S}$ equals $mB$, and each integer $a_i$  satisfies $B/4< a_i< B/2$. The question is whether it is feasible to partition set $\mathcal{S}$ into $m$ subsets, with each subset summing to $B$. This, in turn, implies that each subset must consist of exactly three elements.
%We describe a reduction from an instance of the 3-partition problem
We describe a reduction from an instance of the 3-partition problem
to an instance  of the $L_1$ over $L_2$  minimization problem (\ref{L1o2uncon}) that has the  optimal value $\frac{1}{2}\sqrt{n}$ if and only if the former
has an equitable 3-partition.
Consider  the minimization problem
\begin{eqnarray}\label{3parUn}
\begin{array}{ll}
\min\limits_{{\h u}:={\text{vec}}({\h X})\in \mathbb R^{m n}_+} P(\h u)&=\|\left(({\h e}^{(m)})^\top \otimes I_n\right){\h u}-{\h e}^{(n)}\|^2\\[0.2cm]
&+\|\left({I}_m \otimes {\h a}^\top\right){\h u}-\kappa({\h e}^{(m)})^\top\|^2\\[0.2cm]
&+ \displaystyle{ \frac{1}{2}\frac{\|\h u\|_1}{\|\h u\|_2}}\\[0.4cm]
\mbox{where}\; & {\h x}_j=(x_{1j},\ldots,x_{nj})^\top\in \mathbb R^n_+\\[0.3cm]
& \ \forall\ j\in [m]\\[0.3cm]
 &{\h {X}} = ({\h {x}}_1, \dots, {\h {x}}_m) \in \mathbb{R}_+^{n \times m}.
\end{array}\nn\\
\end{eqnarray}
Note that the objective function value of (\ref{3parUn}) satisfies that
$$P({\h u})\ge\frac{1}{2}\frac{\|\sum_{j=1}^m {\h x}_j\|_1}{\|\sum_{j=1}^m {\h x}_j\|_2}=\frac{1}{2}\sqrt{n}.$$
It achieves the optimal value when $\sum_{i=1}^n a_i x_{ij}=\kappa$ for all $j\in[m]$, and $\sum_{j=1}^m x_{ij}=1$ for all $i\in[n]$, which has only one entry equal to 1, and the others are 0.
 This generates an equitable 3-partition of the ${\cal T}$ entries. On the other hand, if the entries of ${\cal T}$ have an equitable 3-partition, then (\ref{3parUn}) must have  a solution $({\h x}_j)_{j=1}^m$ satisfying $\sum_{i=1}^n a_i x_{ij}=\kappa$ for all $j\in[m]$ and $\sum_{j=1}^m x_{ij}=1$ for all $i\in[n]$ such that $P({\h u}) = \frac{1}{2}\sqrt{n}$. Thus the strong NP-hardness of (\ref{L1o2uncon}) with ${\cal X} = {\mathbb R}_+^n$ is proved.
 The strong NP-hardness of (\ref{L1o2uncon}) with ${\cal X} = {\mathbb R}^n$ can be shown similarly and thus omitted here.
\end{proof}

\section{Hardness of $L_p$ over $L_q$ minimization}\label{Lpq}

In this section, we extend our previous results by exploring more flexible regularization models. We consider the following generalized constrained optimization problem:

\begin{equation}
\begin{aligned}
\min_{\mathbf{x}} & \quad \frac{\|\mathbf{x}\|_{p}}{\|\mathbf{x}\|_{q}} \\[0.2cm]
\text{s.t.} & \quad \mathbf{x} \in \{\mathbf{x} \in \mathcal{X} \mid A \mathbf{x} = \mathbf{b}\},
\end{aligned}
\label{L1o2Conpq}
\end{equation}
where \(\mathcal{X} = \mathbb{R}_+^n\) or \(\mathbb{R}^n\), with \(0 < p \leq 1\) and \(1 < q < +\infty\).

Additionally, we analyze the corresponding unconstrained model:

\begin{equation}
\min_{\mathbf{x} \in \mathcal{X}} \gamma \frac{\|\mathbf{x}\|_{p}}{\|\mathbf{x}\|_{q}} + \frac{1}{2}\|A \mathbf{x} - \mathbf{b}\|_{2}^{2}.
\label{L1o2unconpq}
\end{equation}

By employing the analysis similar to the previous section and invoking Lemma \ref{pqineqnorm}, we can show that both the constrained model \eqref{L1o2Conpq} and the unconstrained model \eqref{L1o2unconpq} are NP-hard and strongly NP-hard. The key results are summarized in the following four theorems and the detailed proof of these results is omitted here.

\begin{theorem}\label{theo0G}
The minimization problem involving the \(L_p\) norm (\(0 < p \leq 1\)) over the \(L_q\) norm (\(1 < q < +\infty\)) for the constrained model \eqref{L1o2Conpq}, with \(\mathcal{X} = \mathbb{R}_+^n\) or \(\mathcal{X} = \mathbb{R}^n\), is NP-hard.
\end{theorem}

\begin{theorem}\label{theo1G}
The minimization problem involving the \(L_p\) norm (\(0 < p \leq 1\)) over the \(L_q\) norm (\(1 < q < +\infty\)) for the constrained model \eqref{L1o2Conpq}, with \(\mathcal{X} = \mathbb{R}_+^n\) or \(\mathcal{X} = \mathbb{R}^n\), is strongly NP-hard.
\end{theorem}

\begin{theorem}\label{theo2G}
The minimization problem involving the \(L_p\) norm (\(0 < p \leq 1\)) over the \(L_q\) norm (\(1 < q < +\infty\)) for the unconstrained model \eqref{L1o2unconpq}, with \(\mathcal{X} = \mathbb{R}_+^n\) or \(\mathcal{X} = \mathbb{R}^n\), is NP-hard.
\end{theorem}

\begin{theorem}\label{theo3G}
The minimization problem involving the \(L_p\) norm (\(0 < p \leq 1\)) over the \(L_q\) norm (\(1 < q < +\infty\)) for the unconstrained model \eqref{L1o2unconpq}, with \(\mathcal{X} = \mathbb{R}_+^n\) or \(\mathcal{X} = \mathbb{R}^n\), is strongly NP-hard.
\end{theorem}

\section{Discussion and conclusion}\label{dis}

The existing computational methods \cite{RWDL19,TaoZhang23, WYYL20, ELX13, YEX14, Tao20, 10542092, ZengYuPong20, BRDL22, Sorted24} for $L_1/L_2$ minimization tend to converge to d-stationary or critical points rather than global solutions. Exact recovery conditions for global solutions of (\ref{L1o2Con}) have been derived in \cite{XU2021486}. This paper presents a formal proof of the difficulty in finding the global optimizer for both (\ref{L1o2Con}) and (\ref{L1o2uncon}).

To transition from a d-stationary or critical point to a local minimizer, the negative curvature algorithm \cite{Curtis2017} can be employed. Furthermore, to progress from a local minimizer to a global solution, the simulated annealing (SA) algorithm \cite{FENGMIN20131577} can be used. However, it is important to note that convergence to a global solution is guaranteed only with probability 1.

In this work, by leveraging our previous results \cite{TaoZhang23}, we have derived a uniform upper bound for any local minimizer of the constrained/unconstrained model in terms of the $L_2$-norm. Additionally, by applying the first- and second-order necessary conditions for a local minimizer, we have established upper/lower bounds for the nonzero entries of any local solution to the unconstrained  model. Moreover, we have shown that the constrained and unconstrained $L_1/L_2$ models are strongly NP-hard. Finally, we have demonstrated that minimizing the $L_p$ norm ($0 < p \leq 1$) over the $L_q$ norm ($1 < q < +\infty$) is also strongly NP-hard.
	%\bibliographystyle{siamplain}
%\bibliography{refer6}

\end{document}